\documentclass[12pt,letterpaper]{scrartcl}

\setkomafont{sectioning}{\normalfont\rmfamily\bfseries}
\setkomafont{section}{\normalsize}
\addtokomafont{caption}{\footnotesize}
\addtokomafont{captionlabel}{\footnotesize}
\setcapindent{1em}
\deffootnote[1em]{0em}{1em}{\textsuperscript{\thefootnotemark}}

\usepackage{amsthm,amsmath,amsfonts,setspace,paralist,mathtools,addlines}
\usepackage[latin1]{inputenc}
\usepackage[pdftex]{color,graphicx}
\usepackage[mathcal]{euscript}
\usepackage{natbib}
\numberwithin{equation}{section}
\allowdisplaybreaks

\newtheoremstyle{thm}{3pt}{3pt}{\itshape}{}{\bfseries}{.}{ }{}
\newtheoremstyle{rest}{3pt}{3pt}{}{}{\bfseries}{.}{ }{}
\newtheoremstyle{comment}{3pt}{3pt}{}{}{\itshape}{.}{ }{}
\theoremstyle{thm}
\newtheorem{theorem}{Theorem}
\newtheorem{lemma}{Lemma}

\newtheorem{proposition}{Proposition}
\newtheorem{assumption}{Assumption}

\theoremstyle{rest}

\newtheorem{example}{Example}
\theoremstyle{comment}
\newtheorem*{comments}{Remarks}
\newtheorem*{comment}{Remark}

\newcommand{\ev}{\mathbb{E}}%
\newcommand{\prob}{\mathbb{P}}%
\DeclareMathOperator{\empro}{\nu}%

\usepackage[bookmarks=false,pdfpagemode=UseNone,pdftex,colorlinks,citecolor=black,%
            filecolor=black,linkcolor=black,urlcolor=black,%
            pdfauthor=Andreas~Hagemann,pdfstartview=FitH]{hyperref}

\title{\normalsize Stochastic Equicontinuity in Nonlinear Time Series Models}
\author{\normalsize Andreas Hagemann\thanks{Department of Economics, University of Notre Dame, 434 Flanner Hall, Notre Dame, IN 46556, USA.
Tel.: +1 (574) 631-1688. Fax: +1 (574) 631-4783.
E-mail address: \href{mailto:andreas.hagemann@nd.edu}{\texttt{andreas.hagemann@nd.edu}}. I would like to thank Roger Koenker and an anonymous referee for helpful comments. I gratefully acknowledge financial support from a UIUC Summer Research Fellowship. All errors are my own.}\\
\normalsize\emph{University of Notre Dame}}
\date{\normalsize\today}

\begin{document}
\maketitle
\begin{abstract}\small
In this paper I provide simple and easily verifiable conditions under which a strong form of stochastic equicontinuity holds in a wide variety of modern time series models. In contrast to most results currently available in the literature, my methods avoid mixing conditions. I discuss several applications in detail.
\end{abstract}


\section{Introduction}
Stochastic equicontinuity typically captures the key difficulty in weak convergence proofs of estimators with non-differentiable objective functions. Precise and elegant methods have been found to deal with cases where the data dependence structure can be described by mixing conditions; see \citet{dedeckeretal2007} for an excellent summary. Mixing assumptions are convenient in this context because they measure how events generated by time series observations---rather than the observations themselves---relate to one another and therefore also measure dependence of functions of such time series. The downside to these assumptions is that they can be hard to verify for a given application. \citet{hansen1996} describes alternatives and considers parametric classes of functions that behave like mixingales, but his results come at the expense of Lipschitz continuity conditions on these functions and rule out many applications of interest.

In this paper I give simple and easily verifiable conditions under which objective functions of econometric estimators are stochastically equicontinuous when the underlying process is a stationary time series of the form
\begin{align}\label{eq:nonlin}
\xi_i = \xi(\varepsilon_{i},\varepsilon_{i-1},\varepsilon_{i-2},\dots).
\end{align}
Here $(\varepsilon_{i})_{\in\mathbb{Z}}$ is a sequence of iid copies of a random variable $\varepsilon$ and $\xi$ is a measurable, possibly unknown function that transforms the input $(\varepsilon_{i},\varepsilon_{i-1},\dots)$ into the output $\xi_i$. The stochastic equicontinuity problem does not have to be parametric and no continuity conditions are needed. The class \eqref{eq:nonlin} allows for the construction of dependence measures that are directly related to the stochastic process and includes a large number of commonly-used stationary time series models. 
The next section provides several specific examples.

In the following, $\Vert X \Vert_p$ denotes $(\ev |X|^p)^{1/p}$ and $\prob^*$ and $\ev^*$ are outer probability and expectation, respectively \citep[see][p.\ 258]{vandervaart1998}. Limits are as $n\to\infty$.

\section{Stochastic Equicontinuity in Nonlinear Time Series Models}\label{sec:equi}

Let $\empro_n f := n^{-1/2}\sum_{i=1}^n \bigl( f(\xi_i) - \ev f(\xi_0) \bigr)$ be the empirical process evaluated at some function $f$.  Here $f$ is a member of a class of real-valued functions $\mathcal{F}$. In econometric applications, $\mathcal{F}$ is typically a parametric class $\{ f_\theta : \theta \in\Theta \}$, where $\Theta$ is a bounded subset of $\mathbb{R}^k$, although no parametric restriction on $\mathcal{F}$ is necessary in the following. Define a norm by $\rho(f) = \Vert f(\xi_0) \Vert_2$. An empirical process is said to be \textit{stochastically equicontinuous} \citep[see, e.g.,][p.\ 139]{pollard1985} on $\mathcal{F}$ if for all $\epsilon > 0$ and $\eta>0$, there is a $\delta>0$ such that 
\begin{equation}
\limsup_{n\to\infty}\prob^* \biggl( \sup_{f,g\in\mathcal{F}:\rho(f-g) < \delta}| \empro_n( f - g)| > \eta \biggr) < \epsilon.\label{eq:stochequi}
\end{equation}

\addlines

As mentioned above, proving stochastic equicontinuity is often the key difficulty in weak convergence proofs. The next four examples illustrate typical applications.
\begin{example}[Quantilograms]\label{ex:quantilo}
\citet{lintonwhang2007} measure the directional predictive ability of stationary time series $(X_i)_{i\in\mathbb{Z}}$ with the quantilogram, a normalized version of $\ev(\alpha - 1\{X_{0} < \theta_\alpha \})(\alpha - 1\{X_{h} < \theta_\alpha \})$ with $\alpha\in(0,1)$ and $h=1,2,\dots$, where $\theta_\alpha$ is the $\alpha$-quantile of the marginal distribution of $(X_i)_{i\in\mathbb{Z}}$. Let $\xi_i = (X_{i-h},X_i)^\top$ and $f_{\theta}(\xi_i) = (\alpha - 1\{X_{i-h} < \theta \})(\alpha - 1\{X_{i} < \theta \})$. Under the null hypothesis of no directional predictability, we have $\ev f_{\theta_\alpha}(\xi_0) = 0$ for all $h=1,2,\dots$. Let $\hat{\theta}_{n,\alpha}$ be the sample $\alpha$-quantile and replace population moments by sample moments to obtain 
$(n-h)^{-1}\sum_{i=1+h}^{n} f_{\hat{\theta}_{n,\alpha}}(\xi_i)$, the sample version of $\ev f_{\theta_\alpha}(\xi_0)$. 
Apart from a scaling factor, the asymptotic null distribution of the sample quantilogram can be determined through the decomposition
\begin{align*}
(n-h)^{-1/2} \sum_{i=1+h}^{n} f_{\hat{\theta}_{n,\alpha}}(\xi_i) = \sqrt{n-h}\,\ev f_{\hat{\theta}_{n,\alpha}}(\xi_0)  + \empro_{n-h} f_{{\theta}_{\alpha}} + \empro_{n-h} (f_{\hat{\theta}_{n,\alpha}} - f_{{\theta}_{\alpha}}).
\end{align*}
If the distribution of $X_i$ is smooth, the delta method can be used to control the first term on the right and, under dependence conditions, an ordinary central limit theorem applies to the second term. Further, we have $\rho(f_{\hat{\theta}_{n,\alpha}} - f_{{\theta}_{\alpha}}) \to_p 0$ whenever $\hat{\theta}_{n,\alpha} \to_p \theta_{\alpha}$ (see Example \ref{ex:quantiloc} below). Hence, we can take $\mathcal{F}=\{ f_\theta : \theta \in\Theta \}$, where $\Theta$ is a compact neighborhood of $\theta_\alpha$, and as long as \eqref{eq:stochequi} holds, the third term on the right-hand side of the preceding display converges to zero in probability because in large samples
\begin{align*}
\prob\Bigl(\empro_{n-h} (f_{\hat{\theta}_{n,\alpha}} - f_{{\theta}_{\alpha}}) > \eta, \rho(f_{\hat{\theta}_{n,\alpha}} - f_{{\theta}_{\alpha}}) < \delta \Bigr) 
\leq \prob^* \biggl( \sup_{f_\theta\in\mathcal{F}:\rho(f_\theta-f_{\theta_\alpha}) < \delta}| \empro_{n-h}( f_\theta - f_{\theta_\alpha})| > \eta \biggr).
\end{align*}
\end{example}

\begin{example}[Robust $M$-estimators of location]\label{ex:huber}
Robust location estimators can often be defined implicitly as an $M$-estimator $\hat{\theta}_n$ that nearly solves $n^{-1}\sum_{i=1}^n f_{\theta}(\xi_i) = 0$ in the sense that $\sum_{i=1}^n f_{\hat{\theta}_n}(\xi_i) = o_p(\sqrt{n})$. Popular examples include the median with $f_{\theta}(x) = \mathrm{sign}(x-\theta)$ and Huber estimators with $f_{\theta}(x) = -\Delta 1\{ x-\theta < -\Delta\} + (x-\theta) 1\{ |x-\theta| \leq \Delta \} + \Delta 1\{ x-\theta > \Delta\}$ for some $\Delta > 0$. Add and subtract in $\sum_{i=1}^n f_{\hat{\theta}_n}(\xi_i) = o_p(\sqrt{n})$ to see that stochastic equicontinuity implies $\sqrt{n}\,\ev f_{\hat{\theta}_n}(\xi_0) + \empro_n f_{\theta_0} = o_p(1).$ The limiting behavior of $\sqrt{n}(\hat{\theta}_n - \theta_0)$ can then again be determined through the delta method and a central limit theorem.
\end{example}

\begin{example}[Stochastic dominance]\label{ex:dominance} When comparing two stationary time series $(X_{i,1})_{i\in\mathbb{Z}}$ and $(X_{i,2})_{i\in\mathbb{Z}}$, $X_{i,1}$ is said to (weakly) stochastically dominate $X_{i,2}$ over $\Theta$ if $\prob(X_{0,1} \leq \theta) \leq \prob(X_{0,2} \leq \theta)$ for all $\theta\in\Theta$. Let $\xi_i = (X_{i,1},X_{i,2})^\top$ and  $f_\theta({\xi_i}) = 1\{X_{i,1}\leq \theta\}-1\{X_{i,2}\leq \theta\}$. The null of stochastic dominance can be expressed as $\sup_{\theta\in\Theta}\ev f_{\theta}(\xi_0)\leq 0$. \citet{lintonetal2005} study weak convergence of the rescaled sample equivalent $\sup_{\theta\in\Theta}n^{-1/2}\sum_{i=1}^n f_{\theta}(\xi_i)$ to the supremum of a certain Gaussian process 
when the null hypothesis is satisfied with $\sup_{\theta\in\Theta}\ev f_{\theta}(\xi_0)= 0$. This convergence follows from the continuous mapping theorem as long as \begin{inparaenum}[(i)]\item\label{ex:dominance_i} $(\Theta,\rho)$ is a totally bounded pseudometric space, \item\label{ex:dominance_ii} $(\nu_n f_{\theta_1},\dots,\nu_n f_{\theta_k})^\top$ converges in distribution for every finite set of points $\theta_1,\dots,\theta_k$ in $\Theta$, and \item $\nu_n f_\theta$ is stochastically equicontinuous; see, e.g., \citet[p.\ 261]{vandervaart1998}. \end{inparaenum}  Condition \eqref{ex:dominance_i} can be shown to hold if $\Theta$ is bounded and $\xi_i$ has Lipschitz continuous marginal distribution functions. Condition \eqref{ex:dominance_ii} can be verified with the help of a multivariate central limit theorem.
\end{example}

\begin{example}[Censored quantile regression]\label{ex:cquantreg}
\citet{volgushevetal2012} develop Bahadur representations for quantile regression processes arising from a linear latent variable model with outcome $T_i$ and covariate vector $Z_i$. Denote the random censoring time by $C_i$. Only $Z_i$, the random event time $\min\{T_i, C_i\}$, and the associated censoring indicator $1\{T_i \leq C_i\}$ are observed. Let $\xi_i = (T_i, C_i, Z_i^\top)^\top$ and $f_\theta(\xi_i) = Z_i 1\{T_i \leq C_i\} 1\{T_i \leq Z_i^{\top}\theta \}$. As \citeauthor{volgushevetal2012}\ point out in their Remark 3.2, a key condition for the validity of their Bahadur representations under dependence is stochastic equicontinuity of $\nu_n f_\theta$.
\end{example}

\addlines

Stochastic equicontinuity cannot hold without restrictions on the complexity of the set $\mathcal{F}$; see, e.g., \citet[][pp.\ 2252--2253]{andrews1994}. Here, complexity of $\mathcal{F}$ is measured via its bracketing number $N = N(\delta,\mathcal{F})$, the smallest number for which there are functions $f_1,\dots,f_N\in\mathcal{F}$ and functions $b_1,\dots,b_N$ (not necessarily in $\mathcal{F}$) such that $\rho(b_k)\leq\delta$ and $|f-f_k|\leq b_k$ for all $1\leq k\leq N$. In addition, some restrictions are required on the memory of the time series. For processes of the form \eqref{eq:nonlin}, the memory is most easily controlled by comparing $\xi_i$ to a slightly perturbed version of itself \citep[see][]{wu2005b}. Let $(\varepsilon_i^*)_{i\in\mathbb{Z}}$ be an iid copy of $(\varepsilon_i)_{i\in\mathbb{Z}}$, so that the difference between $\xi_i$ and $\xi_i' := \xi(\varepsilon_i,\dots,\varepsilon_1,\varepsilon_0^*,\varepsilon_{-1}^*,\dots)$ are the inputs prior to period $1$. Assume the following:
\begin{assumption}\label{as:coupling}
Let $\mathcal{F}$ be a uniformly bounded class of real-valued functions with bracketing numbers $N(\delta,\mathcal{F}) < \infty$. Then there exists some $\alpha\in(0,1)$ and $p > 0$ such that
\begin{compactenum}[\upshape(i)]
\item\label{as:coupling1} $\sup_{f\in\mathcal{F}}\Vert f(\xi_n) - f(\xi_n') \Vert_p = O(\alpha^n)$ and 

\item\label{as:coupling2} $\max_{1\leq k \leq N(\delta,\mathcal{F})}\Vert b_k(\xi_n) - b_k(\xi_n') \Vert_p = O(\alpha^n)$ for any given $\delta >0$.
\end{compactenum}

\end{assumption}
\begin{comments}
\begin{inparaenum}[(i)] 
\item Assumption \ref{as:coupling} is a short-range dependence condition. Proposition \ref{prop:gmc} below presents a device to establish this condition and shows that Assumption \ref{as:coupling} often imposes only mild restrictions on the dependence structure. At the end of this section, I provide a detailed discussion of how to verify this assumption for Examples \ref{ex:quantilo}--\ref{ex:cquantreg}.

\item Because $\mathcal{F}$ is assumed to be uniformly bounded, the bounding functions $b_k$ can be chosen to be bounded as well. Hence, in view of Lemma 2 of \citet{wumin2005}, the exact choice of $p$ is irrelevant, for if Assumption \ref{as:coupling} holds for some $p$, then it holds for all $p > 0$.
\end{inparaenum}
\end{comments}

Assumption \ref{as:coupling} and a complexity requirement on $\mathcal{F}$ given by a bracketing integral imply a strong form of stochastic equicontinuity. The following theorem (the proof of which is found in the Appendix) is similar to \citeauthor{andrewspollard1994}'s (\citeyear{andrewspollard1994}) Theorem 2.2 with their mixing condition replaced by Assumption \ref{as:coupling}. It implies \eqref{eq:stochequi} via the Markov inequality.
\begin{theorem}\label{l:strongequi}
Suppose that Assumption \ref{as:coupling} holds and $\int_0^1 x^{-\gamma/(2+\gamma)}N(x,\mathcal{F})^{1/Q}\, dx < \infty$ for some $\gamma > 0$ and an even integer $Q\geq 2$. Then for every $\epsilon>0$, there is a $\delta>0$ such that 
$$\limsup_{n\to\infty}\ev^* \biggl( \sup_{f,g\in\mathcal{F}:\rho(f-g) < \delta}| \empro_n( f - g)|\biggr)^Q < \epsilon.$$
\end{theorem} 

\begin{comment}
A useful feature of this theorem is that the constants $\gamma$ and $Q$ are not connected to the dependence measures as in \citet{andrewspollard1994}. In contrast to their result, $\gamma$ and $Q$ can therefore be chosen to be as small and large, respectively, as desired to make the bracketing integral converge without restricting the set of time series under consideration. 
\end{comment}

As a referee points out, Assumption \ref{as:coupling} is not primitive and therefore the application at hand determines how difficult it is to verify this assumption. Suitable sufficient conditions can be obtained, e.g., if $\Vert f(\xi_n) - f(\xi_n') \Vert_p$ and $\Vert b_k(\xi_n) - b_k(\xi_n') \Vert_p$ can be uniformly bounded above by constant multiples of $\Vert\xi_n - \xi_n' \Vert_q$ for some $q > 0$. Since the expectation operator is a smoothing operator, these bounds can hold even if $f$ and $b_k$ are not Lipschitz or, more generally, H\"older continuous. Assumption \ref{as:coupling} is then satisfied as long as the geometric moment contraction (GMC) property of \citet{wumin2005} holds, i.e., there is some $\beta \in(0,1)$ and $q>0$ such that $\Vert \xi_n - \xi_n' \Vert_q = O(\beta^n)$. Time series models with the GMC property include, among many others, stationary (causal) ARMA, ARCH, GARCH, ARMA-ARCH, ARMA-GARCH, asymmetric GARCH, generalized random coefficient autoregressive, and quantile autoregressive models; see \citet{shaowu2007} and \citet{shao2011} for proofs and more examples.

\addlines

The problem in Examples \ref{ex:quantilo}--\ref{ex:cquantreg} and in a variety of other applications is the appearance of one or more indicator functions that cause kinks or discontinuities in the objective function. The following result \citep[a generalization of Proposition 3.1 of][]{hagemann2011} combines the GMC property and smoothness conditions on the distribution of the underlying stochastic process to generate the kinds of bounds needed for Assumption \ref{as:coupling} when indicator functions are present. I discuss the result in the examples below. Here and in the remainder of the paper, if $\xi_i$ is vector-valued and has a subvector $X_i$, then the corresponding subvector of the perturbed version $\xi_i'$ is denoted by $X_i'$.

\begin{proposition}\label{prop:gmc}
Suppose $(\xi_i)_{i\in\mathbb{Z}} = (U_i,V_i^\top,W_i^\top)^\top_{i\in\mathbb{Z}}$ has the form \eqref{eq:nonlin} and $\xi$ takes values in $\mathcal{U}\times \mathcal{V}\times \mathcal{W}\subseteq \mathbb{R}\times\mathbb{R}^l\times\mathbb{R}^m$. Assume that uniformly in $w\in\mathcal{W}$, $\prob(U_i \leq x\mid W_i = w)$ is Lipschitz in $x$ on an open interval containing $\mathcal{X}=\{v^\top\lambda \mid v\in\mathcal{V},\lambda\in\Lambda\subseteq \mathbb{R}^l\}$. Suppose $(U_i,V_i^\top)^\top_{i\in\mathbb{Z}}$ has the GMC property. If either \begin{inparaenum}[\upshape(i)] \item\label{prop:gmc_single} $\mathcal{V}$ is a singleton, or \item\label{prop:gmc_func} there is a measurable function $g\colon \mathcal{W}\to\mathcal{V}$ with $g(W_i) = V_i$ and $\Lambda$ is bounded, or \item\label{prop:gmc_indep} $U_i$ and $V_i$ are independent conditional on $W_i$ and $\Lambda$ is bounded, then    $\sup_{\lambda\in\Lambda}\Vert 1\{ U_n < V_n^{\top}\lambda \} - 1\{ U_n' < V_n'^{\top}\lambda \}\Vert_p = O(\alpha^n)$ and $\sup_{\lambda\in\Lambda}\Vert 1\{ U_n \leq V_n^{\top}\lambda \} - 1\{ U_n' \leq V_n'^{\top}\lambda \}\Vert_p = O(\alpha^n)$ for some $\alpha\in(0,1)$ and all $p > 0$.\end{inparaenum}
\end{proposition}

Before concluding this section, the next four examples illustrate how to apply \mbox{Theorem \ref{l:strongequi}} and how to verify Assumption \ref{as:coupling} in practice. 

\begin{example}[Quantilograms, continued]\label{ex:quantiloc}
Suppose $F_X(\theta) := \prob(X_0 \leq \theta)$ is Lipschitz on an open interval containing $\Theta$. Take a grid of points $\min \Theta := t_0 < t_1 < \cdots < t_N =: \max\Theta$ and let $b_k(\xi_i) = 1\{ X_{i-h} < t_k \} - 1\{ X_{i-h} < t_{k-1} \} + 1\{ X_{i} < t_k \} - 1\{ X_{i} < t_{k-1} \}$. Given a $\theta \in \Theta$, we can then find an index $k$ such that $|f_{\theta} - f_{t_{k}}|\leq b_k$, where I used the fact that $|\alpha - 1\{\cdot \}| \leq \max\{\alpha, 1-\alpha\} < 1$. Moreover, by stationarity
$\rho( b_k) \leq 2 \Vert 1\{X_0 < t_{k}\} -  1\{X_0 < t_{k-1}\} \Vert_2 \leq 2 \sqrt{ F_X(t_k) - F_X(t_{k-1}) },$
which is bounded above by a constant multiple of $\sqrt{t_k - t_{k-1}}$ due to Lipschitz continuity. Hence,  if $\rho( b_k) \leq \delta$ for all $1\leq k\leq N$, we can choose bracketing numbers with respect to $\rho$ of order $N(\delta,\mathcal{F}) = O(\delta^{-2})$ as $\delta\to 0$ \citep[see][pp.\ 270--272]{andrewspollard1994,vandervaart1998} and the bracketing integral converges, e.g., for $\gamma = 1$ and $Q=4$. By the same calculations as for $\rho( b_k)$, all $\theta,\theta'\in \Theta$ satisfy $\rho(f_\theta-f_{\theta'}) =O( |\theta-\theta'|^{1/2})$ as $\theta\to\theta'$ and therefore $\rho(f_{\hat{\theta}_{n,\alpha}} - f_{{\theta}_{\alpha}}) \to_p 0$ if $\hat{\theta}_{n,\alpha} \to_p \theta_{\alpha}$. 
In addition, suppose the GMC property holds. Because
$\Vert f_\theta(\xi_n) - f_\theta(\xi_n') \Vert_p \leq 2^{1/p}\Vert 1\{X_n < \theta\} -  1\{X'_n < \theta\} \Vert_p + 2^{1/p}\Vert 1\{X_{n-h} < \theta\} -  1\{X'_{n-h} < \theta\} \Vert_p$ uniformly in $\theta$ by Lo\`eve's $c_r$ inequality, apply Proposition \ref{prop:gmc}\eqref{prop:gmc_single} twice with $U_i=X_i$ and $U_i=X_{i-h}$ to see that Assumption \ref{as:coupling}\eqref{as:coupling1} is also satisfied. In both cases we can take $\Lambda = \Theta$, $V_i\equiv 1$, and $W_i\equiv 0$ (say). The same reasoning applies to $b_k$.
\end{example}

\addlines

\begin{example}[Robust $M$-estimators of location, continued]
Nearly identical arguments as in the preceding example yield stochastic equicontinuity for the median. For the Huber estimator, take the grid from before and note that we can find a $k$ such that $|f_{\theta} - f_{t_k}| \leq \min\{t_k-t_{k-1},2\Delta\} =: b_k$. A routine argument \citep[][Example 19.7, pp.\ 270--271]{andrewspollard1994,vandervaart1998} yields bracketing numbers of order $N(\delta,\mathcal{F}) = O(\delta^{-1})$ as $\delta\to 0$; the bracketing integral is finite, e.g., for $\gamma = 1$ and $Q=2$. Assumption \ref{as:coupling}\eqref{as:coupling1} can be verified via the bound $\sup_{\theta\in\Theta}\Vert f_\theta(\xi_n) - f_\theta(\xi_n') \Vert_p \leq \Vert \xi_n - \xi_n' \Vert_p$ and \eqref{as:coupling2} holds trivially.
\end{example}

\begin{example}[Stochastic dominance, continued]\label{ex:dominancec} 
Replace $(X_{i-h}, X_i)$ in Example \ref{ex:quantiloc} with $(X_{i,1},X_{i,2})$. A slight modification of the arguments presented there to account for the fact that the distribution functions of $X_{i,1}$ and $X_{i,2}$ are not identical yields stochastic equicontinuity of $\nu_n f_\theta$. \citet{lintonetal2005} also consider a more general case with $X_{i,j} = Y_{i,j}-Z_{i,j}^\top\eta_{0,j}$ for $j=1,2$. Here $(X_{i,1},X_{i,2})$ is unobserved and $\eta_0 = (\eta_{0,1}^\top,\eta_{0,2}^\top)^\top$ has to be estimated by some $\hat{\eta}_n = (\hat{\eta}_{n,1}^\top,\hat{\eta}_{n,2}^\top)^\top$. Let $\eta = (\eta_1^\top,\eta_2^\top)^\top$ and $f_{\theta,\eta}(\xi_i) = 1\{Y_{i,1}\leq Z_{i,1}^\top\eta_{1} + \theta\}-1\{Y_{i,2}\leq Z_{i,2}^\top\eta_{2} + \theta\}$. If $\sup_{\theta\in\Theta}\ev f_{\theta,\eta_0}=0$, the asymptotic null distribution of the test statistic $\sup_{\theta\in\Theta}n^{-1/2}\sum_{i=1}^n f_{\theta,\hat{\eta}_n}(\xi_i)$ can be derived from an application of the continuous mapping theorem in
\begin{align*}
n^{-1/2}\sum_{i=1}^n \bigl( f_{\theta,\hat{\eta}_n}(\xi_i) - \ev f_{\theta,\eta_0}(\xi_0)\bigr) = \sqrt{n}\ev (f_{\theta,\hat{\eta}_n}-f_{\theta,\eta_0})(\xi_0) + \nu_n f_{\theta,\eta_0} + \nu_n (f_{\theta,\hat{\eta}_n}-f_{\theta,\eta_0}).
\end{align*}
Under appropriate conditions on $\hat{\eta}_n$, the behavior of the first term on the right follows from the functional delta method; see Lemma 3 and 4 of \citet{lintonetal2005}. The asymptotic properties of the second term are the same as those of $\nu_n f_\theta$ from the simple case where $(X_{i,1},X_{i,2})$ is observed directly. 

\addlines

Verifying stochastic equicontinuity of $\nu_n f_{\theta,\eta}$ to control the third term is more involved. Let $\mathrm{H}_1\times\mathrm{H}_2$ be a bounded neighborhood of $(\eta_{0,1}^\top,\eta_{0,2}^\top)^\top$ and assume the conditional distribution functions of $X_{i,j}$ given $Z_{i,j}$, $j=1,2$, are Lipschitz continuous. The bounded set $\Theta\times\mathrm{H}_1\times\mathrm{H}_2$ can then be covered by balls with suitably chosen radius $r = r(\delta)$ and centers $(t_k,e_{k,1}^\top,e_{k,2}^\top)^\top$, $1\leq k \leq N$, so that the bracketing integral in Theorem \ref{l:strongequi} converges \citep[][Lemma 1]{lintonetal2005}. Here we take $\xi_i = (X_{i,1}, Z_{i,1}^\top,X_{i,2}, Z_{i,2}^\top)^\top$ and
\begin{align*}
b_k(\xi_i) &= 1\{ X_{i,1} < Z_{i,1}^\top(e_{k,1} - \eta_{0,1}) + t_k + (|Z_{i,1}|+1)r \}\\ &\qquad- 1\{ X_{i,1} \leq Z_{i,1}^\top(e_{k,1} - \eta_{0,1}) + t_k - (|Z_{i,1}|+1)r \}\\
&\qquad +1\{ X_{i,2} < Z_{i,2}^\top(e_{k,2} - \eta_{0,2}) + t_k + (|Z_{i,2}|+1)r \}\\ &\qquad- 1\{ X_{i,2} \leq Z_{i,2}^\top(e_{k,2} - \eta_{0,2}) + t_k - (|Z_{i,2}|+1)r \},
\end{align*}
where $|\cdot|$ is Euclidean norm. For fixed $k$ and $r$, define a function $g$ such that the first term on the right-hand side of the preceding display equals $1\{X_{i,1} < g(Z_{i,1})\}$. Assume $\xi_i$ has the GMC property. By the reverse triangle and Cauchy-Schwartz inequalities, $g(Z_i)$ has this property as well. Apply Proposition \ref{prop:gmc}\eqref{prop:gmc_func} with $U_i = X_{i,1}$, $V_i = g(Z_{i,1})$, $W_i = Z_{i,1}$, and $\Lambda = \{1\}$ to establish $\Vert 1\{X_{n,1} < g(Z_{n,1})\} - 1\{X_{n,1}' < g(Z_{n,1}')\} \Vert_p=O(\alpha^n)$. The same holds for the remaining indicator functions with appropriate choices of $g$. Conclude $\Vert b_k(\xi_n) - b_k(\xi_n') \Vert_p = O(\alpha^n)$ for \mbox{all $1\leq k\leq N$,} which verifies Assumption \ref{as:coupling}\eqref{as:coupling2}. Further, write $1\{Y_{i,j} \leq Z_{i,j}^\top\eta_{j} + \theta\} = 1\{X_{i,j} \leq Z_{i,j}^\top(\eta_{j}-\eta_{0,j}) + \theta\}$ so that for all $p>0$ we can bound $\sup_{\theta,\eta}\Vert f_{\theta,\eta}(\xi_n)-f_{\theta,\eta}(\xi_n')\Vert_p$, $\theta\in\Theta$ and $\eta\in\mathrm{H}_1\times\mathrm{H}_2$, by a constant multiple of
\begin{align*}
&\sup_{\theta,\eta_1}\Vert 1\{X_{n,1}\leq Z_{i,1}^\top(\eta_{1}-\eta_{0,1}) + \theta\}-1\{X_{i,1}'\leq Z_{n,1}'^\top(\eta_{1}-\eta_{0,1})+ \theta\}\Vert_p\\
&\qquad+\sup_{\theta,\eta_2}\Vert 1\{X_{n,2}\leq Z_{n,2}^\top(\eta_{2}-\eta_{0,2}) + \theta\}-1\{X_{n,2}'\leq Z_{n,2}'^\top(\eta_{2}-\eta_{0,2}) + \theta\}\Vert_p.
\end{align*}
View the two suprema in the display as suprema over $\lambda = (\theta,\lambda_j^\top)^\top\in\Lambda= \Theta\times \{\eta_j - \eta_{0,j}: \eta_j\in\mathrm{H}_j\}$ for $j=1,2$ and apply Proposition \ref{prop:gmc}\eqref{prop:gmc_func} twice with $U_{i} = X_{i,j}$, $V_i = (1,Z_{i,j}^\top)^\top$, $W_i = Z_{i,j}$ to see that Assumption \ref{as:coupling}\eqref{as:coupling1} holds as well.
\end{example}

\begin{example}[Censored quantile regression, continued] Following \citet[condition C1]{volgushevetal2012}, assume there is some fixed $\Delta > 0$ with $|Z_i|\leq \Delta$. Add and subtract to see that $\sup_{\theta\in\Theta}\Vert f_\theta(\xi_n) - f_\theta(\xi_n')\Vert_p$ does not exceed a constant multiple of
\begin{align*}
\Vert Z_n - Z_n'\Vert_p + \Delta \sup_{\theta\in\Theta}\Vert 1\{T_n < Z_n^\top\theta\} -1\{T_n' < Z_n'^\top\theta\}\Vert_p + \Delta \Vert 1\{T_n < C_n\} -1\{T_n' < C_n'\}\Vert_p.
\end{align*}
If $\xi_i = (T_i,C_i,Z_i^\top)^\top$ has the GMC property, then the first term is of size $O(\alpha^n)$ by \citeauthor{wumin2005}'s (\citeyear{wumin2005}) Lemma 2. Proposition \ref{prop:gmc}\eqref{prop:gmc_func} can be used to establish the same order of magnitude for the second term as long as the conditional distribution of $T_i$ given $Z_i$ satisfies the smoothness condition stated in the proposition. In addition, \citeauthor{volgushevetal2012}\ assume that $T_i$ and $C_i$ are independent conditional on $Z_i$. The size of the third term is then also $O(\alpha^n)$ by Proposition \ref{prop:gmc}\eqref{prop:gmc_indep}. This verifies Assumption \ref{as:coupling}\eqref{as:coupling1}. An argument similar to the one provided in Example \ref{ex:dominancec} establishes Assumption \ref{as:coupling}\eqref{as:coupling2}.

\end{example}

\phantomsection
\addcontentsline{toc}{section}{References}
\addlines[2]
{ 
\bibliography{qspec}}

\begin{thebibliography}{}

\bibitem[\protect\citeauthoryear{Andrews}{Andrews}{1994}]{andrews1994}
Andrews, D. W.~K. (1994).
\newblock Empirical process methods in econometrics.
\newblock In R.~F. Engle and D.~L. McFadden (Eds.), {\em Handbook of
  Econometrics}, Volume~IV, Chapter~37, pp.\  2248--2294. Elsevier.

\bibitem[\protect\citeauthoryear{Andrews and Pollard}{Andrews and
  Pollard}{1994}]{andrewspollard1994}
Andrews, D. W.~K. and D.~Pollard (1994).
\newblock An introduction to functional central limit theorems for dependent
  stochastic processes.
\newblock {\em International Statistical Review\/}~{\em 62}, 119--132.

\bibitem[\protect\citeauthoryear{Dedecker, Doukhan, Lang, Le\'on, Louhichi, and
  Prieur}{Dedecker et~al.}{2007}]{dedeckeretal2007}
Dedecker, J., P.~Doukhan, G.~Lang, J.~R. Le\'on, S.~Louhichi, and C.~Prieur
  (2007).
\newblock {\em Weak Dependence: With Examples and Applications}.
\newblock Springer.

\bibitem[\protect\citeauthoryear{Hagemann}{Hagemann}{2011}]{hagemann2011}
Hagemann, A. (2011).
\newblock Robust spectral analysis.
\newblock Unpublished manuscript, Department of Economics, University of
  Illinois,
  \href{http://arxiv.org/abs/1111.1965v1}{\texttt{arXiv:1111.1965v1}}.

\bibitem[\protect\citeauthoryear{Hansen}{Hansen}{1996}]{hansen1996}
Hansen, B.~E. (1996).
\newblock Stochastic equicontinuity for unbounded heterogeneous arrays.
\newblock {\em Econometric Theory\/}~{\em 12}, 347--359.

\bibitem[\protect\citeauthoryear{Linton, Maasoumi, and Whang}{Linton
  et~al.}{2005}]{lintonetal2005}
Linton, O., E.~Maasoumi, and Y.-J. Whang (2005).
\newblock Consistent testing for stochastic dominance under general sampling
  schemes.
\newblock {\em Review of Economic Studies\/}~{\em 72}, 735--765.

\bibitem[\protect\citeauthoryear{Linton and Whang}{Linton and
  Whang}{2007}]{lintonwhang2007}
Linton, O. and Y.-J. Whang (2007).
\newblock The quantilogram: With an application to evaluating directional
  predictability.
\newblock {\em Journal of Econometrics\/}~{\em 141}, 250--282.

\bibitem[\protect\citeauthoryear{Pollard}{Pollard}{1985}]{pollard1985}
Pollard, D. (1985).
\newblock {\em Convergence of Stochastic Processes}.
\newblock Springer.

\bibitem[\protect\citeauthoryear{Shao}{Shao}{2011}]{shao2011}
Shao, X. (2011).
\newblock Testing for white noise under unknown dependence and its applications
  to goodness-of-fit for time series models.
\newblock {\em Econometric Theory\/}~{\em 27}, 312--343.

\bibitem[\protect\citeauthoryear{Shao and Wu}{Shao and Wu}{2007}]{shaowu2007}
Shao, X. and W.~B. Wu (2007).
\newblock Asymptotic spectral theory for nonlinear time series.
\newblock {\em Annals of Statistics\/}~{\em 35}, 1773--1801.

\bibitem[\protect\citeauthoryear{van~der Vaart}{van~der
  Vaart}{1998}]{vandervaart1998}
van~der Vaart, A.~W. (1998).
\newblock {\em Asymptotic Statistics}.
\newblock Cambridge University Press.

\bibitem[\protect\citeauthoryear{Volgushev, Wagener, and Dette}{Volgushev
  et~al.}{2012}]{volgushevetal2012}
Volgushev, S., J.~Wagener, and H.~Dette (2012).
\newblock Censored quantile regression processes under dependence and
  penalization.
\newblock Unpublished manuscript, Ruhr-Universit{\"a}t Bochum.

\bibitem[\protect\citeauthoryear{Wu}{Wu}{2005}]{wu2005b}
Wu, W.~B. (2005).
\newblock Nonlinear system theory: Another look at dependence.
\newblock {\em Proceedings of the National Academy of Sciences\/}~{\em 102},
  14150--14154.

\bibitem[\protect\citeauthoryear{Wu and Min}{Wu and Min}{2005}]{wumin2005}
Wu, W.~B. and W.~Min (2005).
\newblock On linear processes with dependent observations.
\newblock {\em Stochastic Processes and their Applications\/}~{\em 115},
  939--958.

\end{thebibliography}

\appendix
\section*{Appendix}

\section{Proofs}

\begin{proof}[Proof of Theorem \ref{l:strongequi}]
This follows from a simple modification of \citeauthor{andrewspollard1994}'s (\citeyear{andrewspollard1994}) proof of their Theorem 2.1. The proof requires three steps: \begin{inparaenum}[(i)] \item Their ``Proof of inequality (3.2),'' \item their ``Proof of inequality (3.3),'' and \item their ``Comparison of pairs'' argument. \end{inparaenum} Replace their $i$ with $k$ and their $\tau(h_i)$ with $\tau(b_k)$; then apply Lemma \ref{l:moments} below instead of \citeauthor{andrewspollard1994}'s (\citeyear{andrewspollard1994}) Lemma 3.1 in the derivation of their inequality (3.5) to deduce
$\Vert \max_{1\leq k\leq N} |\empro_n b_k| \Vert_Q \leq C' N^{1/Q} \max\{n^{-1/2}, \max_{1\leq k\leq N} \tau(b_k)\}$
and use this in (i) instead of their inequality (3.5). Another application of the lemma establishes the required analogue of their inequality (3.5) used in (ii). The same inequality can also be applied in (iii). The other arguments remain valid without changes. 
\end{proof}

\begin{lemma}\label{l:moments}
Let $\tau(f) := \rho(f)^{2/(2+\gamma)}$ for some $\gamma > 0$ and suppose that Assumption \ref{as:coupling} holds. For all $n\in\mathbb{N}$, all $f,g\in\mathcal{F}$, and every even integer $Q\geq 2$ we have
\begin{align*}
\ev| \empro_n( f-g)|^Q \leq n^{-Q/2} C\Bigl(\bigl(\tau(f - g)^2 n\bigr) + \dots + \bigl(\tau(f - g)^2 n\bigr)^{Q/2}\Bigr),
\end{align*}
where $C$ depends only on $Q$, $\gamma$, and $\alpha$. The inequality remains valid when $f-g$ is replaced by $b_k$ for any given $k \geq 1$.
\end{lemma} 

\addlines[2]

\begin{proof}[Proof of Lemma \ref{l:moments}] 
Let $Z(i) := f(\xi_i) - \ev f(\xi_0)  - (g(\xi_i)- \ev g(\xi_0))$. Assume without loss of generality that $|Z(i)|\leq 1$ for all $i\geq 1$; otherwise rescale and redefine $C$. Define $Z'(i) = f(\xi_i') - \ev f(\xi_0)  - (g(\xi_i')- \ev g(\xi_0))$ and note that $\ev Z(i) = \ev Z'(i) = 0$ for all $i\in\mathbb{Z}$ and all $f,g\in\mathcal{F}$ because $\xi_i$ and $\xi_i'$ are identically distributed. For fixed $k\geq 2$, $d \geq 1$, and $1\leq m < k$, consider integers $i_1\leq \dots \leq i_m \leq i_{m+1} \leq \dots \leq i_k$ so that $i_{m+1}-i_m = d$. Since $Z(i)$ and $Z'(i)$ are stationary, repeatedly add and subtract to see that 
\begin{align}
\Bigl|\ev Z&({i_1})Z({i_2})\cdots Z({i_{k}}) - \ev Z({i_1})Z({i_2})\cdots Z({i_{m}}) \ev Z({i_{m+1}})\cdots Z({i_{k}})\Bigr|\nonumber\\
&=\Bigl|\ev Z({i_1-i_m})Z({i_2-i_m})\cdots Z({i_{k}-i_m})\nonumber\\ 
&\qquad - \ev Z({i_1-i_m}) Z({i_2-i_m})\cdots Z({0}) \ev Z({d})\cdots Z({i_{k}-i_m})\Bigr|\nonumber\\
&\leq \Bigl|\ev Z({i_1-i_m})\cdots  Z({0})\bigl(Z({d})-Z'({d})\bigr)Z({i_{m+2}-i_m})\cdots Z({i_{k}-i_m})\Bigr|\nonumber\\
&\qquad + \sum_{j=2}^{k-m-1} \Bigl|\ev Z({i_1-i_m})\cdots Z({0}) Z'({d})\times\cdots\nonumber\\
&\qquad\qquad\times \bigl(Z({i_{m+j}-i_m})-Z'({i_{m+j}-i_m})\bigr)\cdots Z({i_{k}-i_m})\Bigr|\nonumber\\
&\qquad +  \Bigl|\ev Z({i_1-i_m})\cdots Z({0}) Z'({d})\cdots Z'({i_{k}-i_m})\nonumber\\  
&\qquad\qquad- \ev Z({i_1-i_m})\cdots Z({0}) \ev Z({d})\cdots Z({i_{k}-i_m})\Bigr| \label{eq:gmcbound}
\end{align}
In particular, the last summand on the right-hand side is zero because $Z({i_1-i_m})\cdots Z({0})$ and $Z'({d})\cdots Z'({i_{k}-i_m})$ are independent and $Z({d})\cdots Z({i_{k}-i_m})$ and $Z'({d})\cdots Z'({i_{k}-i_m})$ are identically distributed. For a large enough $M>0$ and some $s > 1$, Assumption \ref{as:coupling}\eqref{as:coupling1} and distributional equivalence of $Z({d})$ and $Z'({d})$ imply $\Vert Z({d})-Z'({d})\Vert_s \leq \Vert f(\xi_d) - f(\xi'_d)\Vert_s + \Vert g(\xi_d) - g(\xi'_d) \Vert_s \leq 2 \sup_{f\in\mathcal{F}} \Vert f(\xi_d) - f(\xi'_d)\Vert_s \leq M \alpha^d$. H\"older's inequality then bounds the first term on the right-hand side of the preceding display by
\begin{align}\label{eq:holderbound}
\Vert Z({i_1})\cdots Z({i_{m}}) \Vert_p \Vert Z({i_{m+2}})\cdots Z({i_{k}}) \Vert_q M \alpha^d,
\end{align}
where the reciprocals of $p$, $q$, and $s$ sum to $1$. Proceeding similarly to \citet{andrewspollard1994}, another application of the H\"older inequality yields 
$\Vert Z({i_1})\cdots Z({i_{m}}) \Vert_p \leq \bigl(\prod_{j=1}^m \ev |Z({i_j})|^{mp} \bigr)^{1/(mp)} \leq \tau(f-g)^{(2+\gamma)/p}$
whenever $mp\geq 2$ and similarly $\Vert Z({i_{m+2}})\times\cdots\times Z({i_{k}}) \Vert_q \leq \tau(f-g)^{(2+\gamma)/q}$ whenever $(k-m-1)q \geq 2$. Suppose for now that $k\geq 3$. If $k > m+1$, take $s = (\gamma + Q)/\gamma$ and $mp = (k-m-1)q=(k-1)/(1-1/s)$. Decrease the resulting exponent of $\tau(f-g)$ from $Q(2+\gamma)/(Q+\gamma)$ to $2$ so \eqref{eq:holderbound} is bounded by $M \alpha^d \tau(f-g)^2$. If $k\geq 2$ and $k=m+1$, the factor $\Vert Z({i_{m+2}})\cdots Z({i_{k}}) \Vert_q$ is not present in \eqref{eq:holderbound}, but we can still choose $s = (\gamma + Q)/\gamma$ and $mp = (k-1)/(1-1/s)$ to obtain the same bound. Identical arguments also apply to each of the other summands in \eqref{eq:gmcbound}. Hence, we can find some $M'>0$ so that 
\begin{align*}
|\ev Z({i_1})Z({i_2})\cdots Z({i_{k}})| \leq | \ev Z({i_1})Z({i_2})\cdots Z({i_{m}}) \ev Z({i_{m+1}})\cdots Z({i_{k}})| + M' \alpha^d \tau(f-g)^2.
\end{align*}
Here $M'$ in fact depends on $k$, but this does not disturb any of the subsequent steps.

Now replace (A.2) in \citet{andrewspollard1994} by the inequality in the preceding display. In particular, replace their $8\alpha(d)^{1/s}$ with $M' \alpha^d$ and their $\tau^2$ with $\tau(f-g)^2$. The rest of their arguments now go through without changes.
The desired result for $b_k$ follows mutatis mutandis: Simply define $Z(i) = b_k(\xi_i)$, repeat the above steps, and invoke Assumption \ref{as:coupling}\eqref{as:coupling2} in place of Assumption \ref{as:coupling}\eqref{as:coupling1}.
\end{proof}

\addlines

\begin{proof}[Proof of Proposition \ref{prop:gmc}]
Suppose the conditions for \eqref{prop:gmc_indep} hold. Add and subtract, then use $|1\{a < b\} - 1\{c < b\}| \leq 1\{| a-b | \leq |a-c|  \}$ for all $a,b,c\in\mathbb{R}$ and the Cauchy-Schwartz inequality to see that $|1\{ U_n < V_n^{\top}\lambda \} - 1\{ U_n' < V_n'^{\top}\lambda \}|$ is at most 
\begin{align}\label{eq:indbound}
1\{ |U_n - V_n^{\top}\lambda| \leq |U_n-U_n'| \} + 1\{ |U_n' - V_n'^{\top}\lambda| \leq |\lambda || V_n-V_n'| \},
\end{align}
 where $|\cdot|$ is Euclidean norm. 
 
Consider the second term in the preceding display. Denote the distribution function of $U_i$ conditional on $W_i$ by $F_{U\mid W}$. By assumption, we can always find a large enough $n^*$ such that for all $n\geq n^*$ the set $\{ x \pm |\lambda|\beta^{{nq}/{(1+q)}}  : x\in\mathcal{X}, \lambda\in\Lambda \}$ is contained in the interval on which $F_{U\mid W}$ is Lipschitz. Apply the Markov inequality, the GMC property, conditional independence, and continuity to write 
\begin{align*}
&\Vert 1\{ |U_n' -  V_n'^{\top}\lambda| \leq |\lambda | | V_n-V_n'| \} \Vert_p^p\\ 
&\qquad\leq \prob(|U_n - V_n^{\top}\lambda| \leq|\lambda | \beta^{\frac{nq}{1+q}}) + \ev| V_n-V_n'|^q \beta^{-{nq^2}/{(1+q)}}\\
&\qquad= \ev\bigl(\prob(|U_n - V_n^{\top}\lambda| \leq |\lambda |\beta^{\frac{nq}{1+q}}\mid V_n,W_n) \bigr) + O(\beta^{\frac{nq}{1+q}})\\
&\qquad= \ev\bigl( F_{U\mid W}( V_0^{\top}\lambda +|\lambda | \beta^{\frac{nq}{1+q}}\mid W_0) - F_{U\mid W}( V_0^{\top}\lambda - |\lambda |\beta^{\frac{nq}{1+q}}\mid W_0)\bigr) + O(\beta^{\frac{nq}{1+q}})\\
&\qquad= O\bigl(\beta^{\frac{nq}{1+q}}(|\lambda|+1)\bigr).
\end{align*}
The first term in \eqref{eq:indbound} satisfies $\Vert 1\{ |U_n -  V_n^{\top}\lambda| \leq |U_n-U_n'| \} \Vert_p^p = O(\beta^{{nq}/{(1+q)}})$ for the same reasons. Take $\alpha = \beta^{q/(p+pq)}$ and combine these bounds via the Lo\`eve $c_r$ inequality to establish the desired result. 

Assertion \eqref{prop:gmc_func} also follows because the above arguments remain valid with $V_i = g(W_i)$. For \eqref{prop:gmc_single}, we have $V_i\equiv V_i'$ and hence $|1\{ U_n < V_n^{\top}\lambda \} - 1\{ U_n' < V_n'^{\top}\lambda \}|\leq 1\{ |U_n - V_n^{\top}\lambda| \leq |U_n-U_n'| \}$. The second term in \eqref{eq:indbound} is then no longer present, which removes the boundedness restriction on $\Lambda$. Finally, note that all of the above remains valid when strong inequalities are replaced by weak inequalities.
\end{proof}

\end{document}